\documentclass[12pt]{article}

\usepackage[]{amsmath, amssymb,amsfonts,amsthm,mathtools}
\usepackage{tikz}
\usepackage{color}
\usepackage{enumerate}

\theoremstyle{plain}
\newtheorem{theorem}{Theorem}[section]
\newtheorem{lemma}[theorem]{Lemma}
\newtheorem{proposition}[theorem]{Proposition}
\newtheorem{corollary}[theorem]{Corollary}

\theoremstyle{definition}
\newtheorem{definition}[theorem]{Definition}

\theoremstyle{remark}

\newcommand\A{\mathcal{A}}

\newcommand\pgma{\mathcal{M}_{G,\psi}}

\newcommand\psitilde{\widetilde{\psi}}

\DeclareMathOperator{\Der}{Der}

\title {Vertex-weighted Graphs and Freeness of $ \psi $-graphical Arrangements 
}
\author{
Daisuke Suyama\thanks{
Department of Mathematics, Hokkaido University, Sapporo, Hokkaido 060-0810, Japan.
email:dsuyama@math.sci.hokudai.ac.jp}, 
Shuhei Tsujie\thanks{Department of Mathematics, Hokkaido University, Sapporo, Hokkaido 060-0810, Japan.
email:tsujie@math.sci.hokudai.ac.jp}
}
\date{}

\begin{document}
\maketitle
	
\begin{abstract}
Let $ G $ be a simple graph on $ \ell $ vertices $ \{1, \dots, \ell \} $ with edge set $ E_{G} $. 
The graphical arrangement $ \mathcal{A}_{G} $ consists of hyperplanes $ \{x_{i}-x_{j}=0\} $, where $ \{i, j \} \in E_{G} $. 
It is well known that three properties, chordality of $ G $, supersolvability of $ \mathcal{A}_{G} $, and freeness of $ \mathcal{A}_{G} $ are equivalent. 
Recently, Richard P. Stanley introduced $ \psi $-graphical arrangement $ \mathcal{A}_{G, \psi} $ as a generalization of graphical arrangements. 
Lili Mu and Stanley characterized the supersolvability of the $ \psi $-graphical arrangements and conjectured that the freeness and the supersolvability of $ \psi $-graphical arrangements are equivalent. 
In this paper, we will prove the conjecture. 
\end{abstract}

{\footnotesize {\it Keywords:} 
Hyperplane arrangement, 
Graphical arrangement, 
Ish arrangement,
Coxeter arrangement, 
Free arrangement,
Supersolvable arrangement,
Chordal graph, 
Vertex-weighted graph,
Multiarrangement
}

{\footnotesize {\it 2010 MSC:}  
52C35, 
32S22,  
05C15, 
05C22, 
20F55,  
13N15 
}

\section{Introduction}\label{sec:introduction}
Let $ V $ be an $ (\ell + 1) $-dimensional vector space over a field $ \mathbb{K} $ and $ \{z, x_{1}, \dots, x_{\ell}\} $ a basis for the dual space $ V^{\ast} $. 
A \textbf{central hyperplane arrangement} (arrangement, for short) in $ V $ is a finite set of vector subspaces of codimension $ 1 $.  

Let $ G $ be a simple graph with vertex set $ V_{G} = \{1, \dots, \ell \} $ and edge set $ E_{G} $. 
Let $ \psi \colon V_{G} \rightarrow 2^{\mathbb{K}} $ be a map satisfying $ |\psi(i)| < \infty $ for every vertex $ i \in V_{G} $. 
The \textbf{graphical arrangement} $ \mathcal{A}_{G} $ and the \textbf{$ \psi $-graphical arrangement} $ \mathcal{A}_{G,\psi} $ are defined by
\begin{align*}
\mathcal{A}_{G} \coloneqq \left\{\{x_{i} - x_{j} = 0\} \mid \{i,j\} \in E_{G} \right\}, 
\end{align*}
\begin{align*}
\mathcal{A}_{G,\psi} \coloneqq \left\{ \{ z=0 \} \right\} \cup \left\{ \{ x_{i}-x_{j} = 0 \} \mid \{i,j\} \in E_{G} \right\} \\
\cup \left\{\{x_{i}=az\} \mid 1 \leq i \leq \ell, a \in \psi(i) \right\}, 
\end{align*}
where $ \{ \alpha = 0 \} \, (\alpha \in V^{\ast}) $ stands for the hyperplane $ \{ v \in V \mid \alpha(v) = 0 \} $.   
When $ G $ is complete, the arrangement $ \mathcal{A}_{G} $ is known as the braid arrangement or the Coxeter arrangement of type $ A_{\ell-1} $. 
The $ \psi $-graphical arrangements were introduced by Richard P. Stanley \cite{stanley2015valid} to investigate the number of chambers of the visibility arrangements of order polytopes. 

For a central arrangement $ \mathcal{A} $, 
the \textbf{intersection lattice} $ L(\mathcal{A}) $ is the set of intersections of hyperplanes in $ \mathcal{A} $ with the order by reverse inclusion: 
$ X \leq Y \Leftrightarrow X \supseteq Y $. 
We say that $ \mathcal{A} $ is \textbf{supersolvable} if the lattice $ L(\mathcal{A}) $ is supersolvable as defined by Stanley \cite{stanley1972supersolvable}. 

Let $ S $ be the symmetric algebra of the dual space $ V^{\ast} $. 
Then $ S $ can be identified with the polynomial ring $ \mathbb{K}[z, x_{1}, \dots, x_{\ell}]. $
Let $\Der(S)$ be the module of derivations of $S$:
\begin{multline*}
\Der(S):=
\{ \theta : S \rightarrow S \mid
\theta \text{ is } \mathbb{K}\text{-linear}, \\
\theta(fg)=f\theta(g) + \theta(f)g \text{ for any } f,g \in S \}.
\end{multline*}
The module of logarithmic derivations $D(\mathcal{A})$ is defined to be
\begin{align*}
D(\mathcal{A}) 
&:=\{ \theta \in \Der(S) \mid \theta (Q(\mathcal{A})) \in Q(\mathcal{A}) S \} \\
&=\{ \theta \in \Der(S) \mid 
\theta (\alpha_{H}) \in \alpha_{H} S \text{ for any } H \in \mathcal{A} \},
\end{align*}
where $Q(\mathcal{A}) \coloneqq \prod_{H \in \mathcal{A}} \alpha_{H} $ is the defining polynomial of $\mathcal{A}$
and $\alpha_{H}$ is a linear form such that $\ker (\alpha_{H}) = H$.
We say that $\mathcal{A}$ is \textbf{free} if $D(\mathcal{A})$ is a free $S$-module.
It is well known that a supersolvable arrangement is free \cite{jambu1984free, orlik1992arrangements}. 
When $ \mathcal{A} $ is free, the module $D(\mathcal{A})$ has a homogeneous basis $\{ \theta_{0}, \ldots ,\theta_{\ell} \}$. 
The tuple of degrees 
$\exp \mathcal{A} = (\deg \theta_{0}, \ldots ,\deg \theta_{\ell})$
is called the \textbf{exponents} of $\mathcal{A}$. 

A graph is called \textbf{chordal} if every cycle of length at least four has a \textbf{chord}, which is an edge that is not part of the cycle but connects two vertices of the cycle. 
A vertex $ v $ is called \textbf{simplicial} if the subgraph induced by its neighbors is complete. 
We say that an ordering of the vertices $ (v_{1}, \dots, v_{\ell}) $ is a \textbf{perfect elimination ordering } if each $ v_{i} $ is simplicial in the subgraph induced by the vertices $ \{v_{1}, \dots, v_{i} \} $. 
\begin{theorem}[Fulkerson-Gross {\cite[Section 7]{fulkerson1965incidence}}]\label{thm:FG}
Let $ G $ be a simple graph. 
Then $ G $ is chordal if and only if $ G $ has a perfect elimination ordering. 
\end{theorem}

These notions for graphs and properties for graphical arrangements are related by the following theorem. 

\begin{theorem}[Stanley {\cite[Corollary 4.10]{stanley2007introduction}}, Edelman-Reiner {\cite[Theorem 3.3]{edelman1994free}}]\label{thm:stanley}
The following three conditions are equivalent: 
\begin{enumerate}
\item $ G $ is chordal. 
\item $ \mathcal{A}_{G} $ is supersolvable. 
\item $ \mathcal{A}_{G} $ is free. 
\end{enumerate}
\end{theorem}

We say that a perfect elimination ordering $ (v_{1}, \dots, v_{\ell}) $ of $ G $ is a \textbf{weighted elimination ordering} if $ \psi(v_{i}) \supseteq \psi(v_{j}) $ whenever $ i < j $ and  $\{v_{i}, v_{j}\} \in E_{G} $. 
Lili Mu and Stanley \cite{mu2015supersolvability} characterized the supersolvability of $ \psi $-graphical arrangements and conjectured that the freeness and the supersolvability of $ \psi $-graphical arrangements are equivalent. 

\begin{theorem}[Stanley {\cite[Theorem 6]{stanley2015valid}} Mu-Stanley {\cite[Theorem 1, 2]{mu2015supersolvability}}]\label{thm:MS1}
A $ \psi $-graphical arrangement $ \mathcal{A}_{G, \psi} $ is supersolvable if and only if $ (G, \psi) $ has a weighted elimination ordering. 
\end{theorem}

We say that a path $ v_{1} \cdots v_{k} $ in $ G $ is \textbf{unimodal}
if there exists $ i \in \{1, \dots, k\} $ such that $ \psi(v_{1}) \subseteq \dots \subseteq \psi(v_{i}) \supseteq \dots \supseteq \psi(v_{k}) $. 
Note that an edge $ \{i, j\} \in E_{G} $ is unimodal if $ \psi(i) $ and $ \psi(j) $ are comparable, i.e., $ \psi(i) \subseteq \psi(j) $ or $ \psi(i) \supseteq \psi(j) $. 

The main result in this paper is as follows: 

\begin{theorem}\label{thm:main1}
For $ \psi $-graphical arrangments $ \mathcal{A}_{G, \psi} $, the following conditions are equivalent: 
\begin{enumerate}
\item $ (G, \psi) $ has a weighted elimination ordering. 
\item $ \mathcal{A}_{G, \psi} $ is supersolvable. 
\item $ \mathcal{A}_{G, \psi} $ is free. 
\item $ G $ is chordal and does not contain the both of the following induced paths: 
\begin{enumerate}
\item[(i)] An edge $ \{v_{1}, v_{2}\} $ such that $ \psi(v_{1}) $ and $ \psi(v_{2}) $ are not comparable. 
\item[(ii)] A path $ v_{1} \cdots v_{k} $ with $ k \geq 3 $ and $ \psi(v_{1}) \supsetneq \psi(v_{2}) = \dots = \psi(v_{k-1}) \subsetneq \psi(v_{k}) $. 
\end{enumerate}
\item $ G $ is chordal and every induced path is unimodal. 
\end{enumerate}
\end{theorem}

The organization of this paper is as follows.
In Section \ref{sec:vertex-weighted graphs}, we introduce a class of vertex-weighted graph over a poset in order to prove the equivalence $ (1) \Leftrightarrow (4) \Leftrightarrow (5) $ in Theorem \ref{thm:main1}. 
In Section \ref{sec:N-Ish}, we will describe a relation between $ \psi $-graphical arrangements and $ N $-Ish arrangements, which are other deformation of braid arrangements. 
In Section \ref{sec:basis}, we will construct a basis for the logarithmic derivation module of the $ \psi $-graphical arrangement when $ (G, \psi) $ has a weighted elimination ordering. 
In Section \ref{sec:proof}, we will complete the proof of Theorem \ref{thm:main1}. 
In Section \ref{sec:multi}, we will introduce the multiarrangements corresponding to vertex-weighted graphs over nonnegative integers and give a characterization of their freeness. 

\section{Vertex-weighted graphs over a poset}\label{sec:vertex-weighted graphs}
Let $ G = (V_{G}, E_{G}) $ be a simple graph with $ \ell $ vertices and $ P $ a poset. 
For a map $ \psi \colon V_{G} \rightarrow P $, we call the pair $ (G, \psi) $ a \textbf{vertex-weighted graph over} $ P $. 
If $ P $ is a singleton then the vertex-weighted graph $ (G, \psi) $ may be identified with the graph $ G $. 
Note that when $ P $ is the poset consisting of finite subsets in $ \mathbb{K} $, the pair $ (G, \psi) $ defines the $ \psi $-graphical arrangement $ \mathcal{A}_{G, \psi} $ as mentioned in Section \ref{sec:introduction}. 

We say that an ordering $ (v_{1}, \dots, v_{\ell}) $ of the vertices in $ (G, \psi) $ is a \textbf{weighted elimination ordering} if $ (v_{1}, \dots, v_{\ell}) $ is a perfect elimination ordering and $ \psi(v_{i}) \supseteq \psi(v_{j}) $ whenever $ i<j $ and $ \{v_{i}, v_{j}\} \in E_{G} $. 
For an induced subgraph $ S $ of $ G $, let $ \psi_{S} $ denote the restriction of $ \psi $ to $ V_{S} $. 
We call the pair $ (S,\psi_{S}) $ an \textbf{induced subgraph} of $ (G,\psi) $. 
Note that a weighted elimination ordering of $ (G,\psi) $ induces a weighted elimination ordering of any induced subgraph of $ (G,\psi) $. 
A path $ v_{1} \cdots v_{k} $ in $ (G, \psi) $ is said to be \textbf{unimodal} if there exists $ i \in \{1, \dots, k\} $ such that $ \psi(v_{1}) \leq \dots \leq \psi(v_{i}) \geq \dots \geq \psi(v_{k}) $. 

\begin{proposition}\label{prop:non-unimodal}
A path in $ (G, \psi) $ is unimodal if and only if it contains none of the following paths: 
\begin{enumerate}[(i)]
\item \label{prop:non-unimodal1} An edge $ \{v_{1}, v_{2}\} $ such that $ \psi(v_{1}) $ and $ \psi(v_{2}) $ are not comparable. 
\item \label{prop:non-unimodal2} A path $ v_{1} \cdots v_{k} $ with $ k \geq 3 $ and $ \psi(v_{1}) > \psi(v_{2}) = \dots = \psi(v_{k-1}) < \psi(v_{k}) $. 
\end{enumerate}
\end{proposition}
\begin{proof}
Since a path of type (\ref{prop:non-unimodal1}) or (\ref{prop:non-unimodal2}) is not unimodal, a path containing at least either one of these paths is not unimodal. 
We will prove the converse. 
Let $ v_{1} \cdots v_{k} $ be a non-unimodal path in $ (G, \psi) $. 
We may assume that there is no edge of type (\ref{prop:non-unimodal1}), i.e.,  $ \psi(v_{i}) $ and $ \psi(v_{i+1}) $ are comparable for every $ i \in \{1, \dots, k-1 \} $. 
Since the path $ v_{1} \cdots v_{k} $ is not unimodal, there exist the minimum index $ i_{0} \in \left\{ 1, \dots, k-2 \right\} $ satisfying $ \psi(v_{i_{0}}) > \psi(v_{i_{0}+1}) $ and the minimum index $ i_{2} \in \{i_{0}+2, \dots, k\} $ satisfying $ \psi(v_{i_{2}-1}) < \psi(v_{i_{2}}) $. 
Let $ i_{1} \in \{ i_{0}, \dots, i_{2}-2 \} $ be the maximum index satisfying $ \psi(v_{i_{1}}) > \psi(v_{i_{1}+1}) $. 
The maximality of $ i_{1} $ and the minimality of $ i_{2} $ imply $ \psi(v_{i_{1}+1}) = \dots = \psi(v_{i_{2}-1}) $. 
Hence the path $ v_{i_{1}} \cdots v_{i_{2}} $ is of type (\ref{prop:non-unimodal2}). 
\end{proof}

We will prove the following theorem, which is a generalization of Theorem \ref{thm:FG}. 
This theorem proves the equivalence $ (1) \Leftrightarrow (4) \Leftrightarrow (5) $ in Theorem \ref{thm:main1}. 

\begin{theorem}\label{thm:vwg}
Let $ (G, \psi) $ be a vertex-weighted graph. 
Then the following are equivalent: 
\begin{enumerate}
\item $ (G, \psi) $ has a weighted elimination ordering. 
\item $ G $ is chordal and does not contain induced paths in Proposition \ref{prop:non-unimodal}. 
\item $ G $ is chordal and every induced path is unimodal.
\end{enumerate} 
\end{theorem}
In order to prove Theorem \ref{thm:vwg}, the following lemma is required. 
\begin{lemma}[Dirac {\cite[Theorem 4]{dirac1961rigid}}]\label{lem:dirac}
Every chordal graph is complete or has at least two non-adjacent simplicial vertices. 
\end{lemma}
\begin{proof}[Proof of Theorem \ref{thm:vwg}]
The equivalence $ (2) \Leftrightarrow (3) $ is clear from Proposition \ref{prop:non-unimodal}. 
In order to prove $ (1) \Rightarrow (2) $, 
assume that $ (G, \psi) $ has a weighted elimination ordering. 
Then $ G $ is chordal by Theorem \ref{thm:FG}. 
Suppose that $ (G, \psi) $ has a path mentioned in Proposition \ref{prop:non-unimodal} as an induced path. 
This path has no weighted elimination ordering, which is a contradiction since every induced subgraph of $ (G, \psi) $ has a weighted elimination ordering. 
Therefore $ G $ has no induced paths in Proposition \ref{prop:non-unimodal}. 

To prove $ (2) \Rightarrow (1) $ suppose that $ G $ is chordal and does not contain induced paths in Proposition \ref{prop:non-unimodal}. 
Then $ \psi(u) $ and $ \psi(v) $ are comparable for every edge $ \{u, v\} \in E_{G} $. 
We will prove that $ (G, \psi) $ has a weighted elimination ordering by induction on $ \ell $. 
If $ \ell = 1 $ then the only ordering is a weighted elimination ordering. 
Assume that $ \ell \geq 2 $. 
It suffices to show that there exists a simplicial vertex $ v $ in $ G $ such that $ \psi(v) \leq \psi(u) $ for any neighbor $ u $ of $ v $ since $ G \setminus \{ v \} $ has a weighted elimination ordering by the induction hypothesis. 
Let $ S $ be a connected component of the subgraph of $ G $ induced by the vertices
\begin{align*}
\left\{ v \in V_{G} \mid  \psi(v) \text{ is minimal in } \psi(V_{G}) \right\}
\end{align*}
Note that for any vertices $ u,v \in V_{S} $ we have that $ \psi(u) = \psi(v) $ by the minimality. 
Let 
\begin{align*}
N \coloneqq \left\{ v \in V_{G}\setminus V_{S} \mid \{ v, u \} \in E_{G} \text{ for some } u \in V_{S} \right\}. 
\end{align*}
We will prove that $ N $ is a clique of $ G $. 
If $ N $ were not a clique then there are non-adjacent vertices $ u, v \in N $. 
For any $ w \in  V_{S} $, we have that $ \psi(u) > \psi(w) < \psi(v) $ by the minimality and comparability. 
Hence a shortest path from $ u $ to $ v $ in the induced subgraph $ S \cup \{u, v\} $ is a path of type (\ref{prop:non-unimodal2}) in Proposition \ref{prop:non-unimodal}, which is a contradiction. 
Thus $ N $ is a clique. 

Let $ F $ be the subgraph of $ G $ induced by $ V_{S} \cup N $. 
Since $ G $ is chordal, $ F $ is also chordal. 
Now we will prove that there exists a vertex $ v $ of $ S $ which is simplicial in $ F $. 
If $ F $ is complete, then every vertex of $ F $ is simplicial in $ F $. 
In particular, every vertex of $ S $ is simplicial in $ F $. 
If $ F $ is not complete, then by Lemma \ref{lem:dirac} there exist two non-adjacent vertices of $ F $ which are simplicial in $ F $. 
Since $ N $ is a clique, at least one of the vertices belongs to $ S $. 
Thus in the both cases, $ S $ has a vertex $ v $ which is simplicial in $ F $. 

Since every neighbor of $ v $ in $ G $ belongs to $ F $, we have that $ v $ is also simplicial in $ G $. 
This is a desired simplicial vertex. 
\end{proof}

\section{Relation to $ N $-Ish arrangements}\label{sec:N-Ish}
Abe and the authors \cite{abe2014freeness} introduced $ N $-Ish arrangements, which are deformation of braid arrangements, to state the sharing property of the freeness of deleted Shi arrangements and deleted Ish arrangements. 

Let $ N = (N_{1}, \dots, N_{\ell}) $ be a tuple of finite subsets in $ \mathbb{K} $. 
The tuple $ N $ is a \textbf{nest} if there exists a permutation $ w $ of $ \{1, \dots, \ell\} $ such that $ N_{w(1)} \subseteq N_{w(2)} \subseteq \dots \subseteq N_{w(\ell)} $. 
Define the $ N $-Ish arrangement $ \mathcal{I}_{N} $ by 
\begin{multline*}
\mathcal{I}_{N} \coloneqq \left\{ \{ z=0 \} \right\} \cup \left\{ \{x_{i}-x_{j} = 0 \} \mid 1 \leq i < j \leq \ell  \right\} \\
\cup \left\{ \{x_{0}-x_{i} = az \} \mid 1 \leq i \leq \ell, \, a \in N_{i}   \right\}, 
\end{multline*}
where $ z, x_{0}, \dots, x_{\ell} $ are coordinates of an $ (\ell+2) $-dimensional vector space over $ \mathbb{K} $. 

\begin{theorem}[{\cite[Theorem 1.3]{abe2014freeness}}]\label{thm:Ish}
The following conditions are equivalent: 
\begin{enumerate}
\item $ N $ is a nest. 
\item $ \mathcal{I}_{N} $ is supersolvable. 
\item $ \mathcal{I}_{N} $ is free. 
\end{enumerate}
\end{theorem}

In \cite{abe2014freeness}, this theorem is formulated for fields of characteristic $0$.
However, the proof is independent of the field. 

\begin{proposition}\label{prop:afequiv}
Let $ G $ be the complete graph with $ \ell $ vertices $ \{1, \dots, \ell\} $ and $ N = (N_{1}, \dots, N_{\ell}) $ a tuple as above. 
Define a map $ \psi \colon V_{G} \rightarrow 2^{\mathbb{K}} $ by $ \psi(i) = N_{i} $ for every $ i \in V_{G} $. 
Then $ \mathcal{A}_{G, \psi} \times \varnothing_{1} $ and $ \mathcal{I}_{N} $ are affinely equivalent, where $ \varnothing_{1} $ denotes the $ 1 $-dimensional empty arrangement. 
\end{proposition}
\begin{proof}
The change of coordinates $ x_{i} \mapsto x_{0}-x_{i} \, (1 \leq i \leq \ell), x_{0} \mapsto x_{0}, z \mapsto z  $ induces the equivalence. 
\end{proof}

Mu and Stanley stated the following lemma without proof. 

\begin{lemma}[Mu-Stanley {\cite[Theorem 3]{mu2015supersolvability}}]\label{lem:MS2}
If a $ \psi $-graphical arrangement $ \mathcal{A}_{G, \psi} $ is free then $ \psi(i) \subseteq \psi(j) $ or $ \psi(i) \supseteq \psi(j) $ for all $ \{i, j\} \in E_{G} $. 
\end{lemma}

In the rest of this section, we will give a proof of Lemma \ref{lem:MS2} with the characterization of the freeness of $ N $-Ish arrangements. 

A subarrangement $ \mathcal{B} $ of an arrangement $ \mathcal{A} $ is called a \textbf{localization} if 
\begin{align*}
\mathcal{B} = \mathcal{A}_{X} \coloneqq \left\{ H \in \mathcal{A} \mid H \supseteq X \right\}
\end{align*}
for some $ X \in L(\mathcal{A}) $.
It is well known that every localization of a free arrangement is also free (see, for example, \cite[Theorem 4.37]{orlik1992arrangements}). 

\begin{proposition}\label{prop:local}
For a $ \psi $-graphical arrangement $ \mathcal{A}_{G, \psi} $, 
let $ S=(V_{S}, E_{S}) $ be an induced subgraph of $ G $ and $ \psi_{S} $ the restriction of $ \psi $ to $ V_{S} $. 
Then the $ \psi_{S} $-graphical arrangement $ \mathcal{A}_{S, \psi_{S}} $ and the graphical arrangement $ \mathcal{A}_{S} $ are localizations of $ \mathcal{A}_{G, \psi} $. 
\end{proposition}
\begin{proof}
First we will prove that $ \mathcal{A}_{S. \psi_{S}} $ is a localization of $ \mathcal{A}_{G,\psi} $. 
Let $ X \coloneqq \bigcap_{H \in \mathcal{A}_{S, \psi_{S}}} H \in L(\mathcal{A}_{G, \psi}) $.
We will show that $ \mathcal{A}_{S, \psi_{S}} = \mathcal{A}_{X} $. 
The inclusion $ \mathcal{A}_{S, \psi_{S}} \subseteq \mathcal{A}_{X} $ is trivial. 
We will show the converse $ \mathcal{A}_{S, \psi_{S}} \supseteq \mathcal{A}_{X} $. 
Take the vector $ v \in X $ such that 
\begin{align*}
z(v) = 0, \quad
x_{i}(v) = \begin{cases}
0 & \text{ if } i \in V_{S}.  \\
i & \text{ otherwise. }
\end{cases}
\end{align*}
Every hyperplane in $ \mathcal{A}_{X} $ must contain the vector $ v $. 
The arrangement $ \mathcal{A}_{G, \psi} $ consists of hyperplanes of three types, $ \{z=0\}, \{x_{i}-x_{j}=0 \} $ and $ \{ x_{i}=az \} $. 
The hyperplane $ \{z=0\} $ is in the both of $ \mathcal{A}_{S, \psi_{S}} $ and $ \mathcal{A}_{X} $. 
If $ \{x_{i}-x_{j}=0\} \in \mathcal{A}_{X} $ then $ x_{i}(v) = x_{j}(v) = 0 $. 
Hence $ i, j \in V_{S} $ and $ \{i, j\} \in E_{S} $ since $ S $ is an induced graph of $ G $. 
Therefore $ \{x_{i}-x_{j} = 0 \} \in \mathcal{A}_{S, \psi_{S}} $. 
If $ \{x_{i} = az \} \in \mathcal{A}_{X} $ then $ x_{i}(v) = az(v) = 0 $. 
Therefore $ i \in V_{S} $ and $ \{x_{i} = az \} \in \mathcal{A}_{S, \psi_{S}} $. 
Hence $ \mathcal{A}_{S, \psi_{S}} \supseteq \mathcal{A}_{X} $. 
Thus $ \mathcal{A}_{S, \psi_{S}} $ is a localization of $ \mathcal{A}_{G, \psi} $. 

To verify that $ \mathcal{A}_{S} $ is a localization, take the vector $ v \in X \coloneqq \bigcap_{H \in \mathcal{A}_{S}} \in L(\mathcal{A}_{G, \psi}) $ satisfying
\begin{align*}
z(v) = -1, \quad 
x_{i}(v) = \begin{cases}
0 & \text{ if } i \in V_{S}. \\
i & \text{ otherwise. }
\end{cases}
\end{align*}
By the similar argument we have $ \mathcal{A}_{S} = \mathcal{A}_{X} $. 
Hence $ \mathcal{A}_{S} $ is a localization. 
\end{proof}

\begin{proof}[Proof of Lemma \ref{lem:MS2}]
Suppose that $ \mathcal{A}_{G, \psi} $ is free and $ \{ i, j \} \in E_{G} $. 
Let $ S $ be the subgraph induced by vertices $ \{ i, j \} $. 
By Proposition \ref{prop:local},  $ \mathcal{A}_{S, \psi_{S}} $ is a localization of $ \mathcal{A}_{G, \psi} $ hence free. 
By Proposition \ref{prop:afequiv}, the arrangement $ \mathcal{A}_{S, \psi_{S}} $ is affinely equivalent to the $ N $-Ish arrangement $ \mathcal{I}_{N} $, where $ N = (\psi(i), \psi(j)) $. 
Theorem \ref{thm:Ish} asserts that $ N $ is a nest. 
Therefore $ \psi(i) \subseteq \psi(j) $ or $ \psi(i) \supseteq \psi(j) $. 
\end{proof}

\section{Basis Construction}\label{sec:basis}
In this section, we construct a basis for 
the logarithmic derivation module of $\mathcal{A}_{G,\psi}$ when $ (G,\psi) $ has a weighted elimination ordering. 

\begin{theorem}
\label{basis}
Suppose that $(G,\psi)$ has a weighted elimination ordering $ (v_{1},\ldots,v_{\ell}) $.
We may assume that $x_{1},\ldots,x_{\ell}$ are the coordinates corresponding to $ v_{1},\ldots,v_{\ell} $.
Define the set $C_{\geq k}$ by
\begin{align*}
C_{\geq k} &\coloneqq
\{ i \mid k < i \leq \ell 
\text{ and there exists a path } 
v_{k}v_{j_{1}}v_{j_{2}} \cdots v_{j_{n}} v_{i} \\
&\hspace{8em} 
\text{ such that }
k < j_{1} < j_{2} < \cdots <j_{n} < i \} \cup \{ k \}.
\end{align*}
Then the homogeneous derivations
\begin{align*}
\theta_{E}&=
\sum_{i=1} x_{i} \frac{\partial}{\partial x_{i}} 
+z\frac{\partial}{\partial z}, \\ 
\theta_{k}&=
\sum_{i \in C_{\geq k}} 
\left(
\prod_{j \in E_{< k}}
(x_{j}-x_{i})
\prod_{a \in \psi (v_{k})}
(x_{i}-az)
\right)
\frac{\partial}{\partial x_{i}}
\ \ (1 \leq k \leq \ell)
\end{align*}
form a basis for $D(\mathcal{A}_{G,\psi})$,
where
$E_{< k}=\{ j \mid 1 \leq j < k,\ \{v_{j},v_{k}\} \in E_{G} \}$.
\end{theorem}

\begin{proof}
First we will prove that the derivations belong to 
$D(\mathcal{A}_{G,\psi})$.
It is known that the Euler derivation $\theta_{E}$
belongs to the logarithmic derivation module of
any central arrangement.
Since $\theta_{k}$ does not contain 
$\frac{\partial}{\partial z}$, we have 
$\theta_{k}(z)=0 \in zS$.

Suppose $\{x_{s}-x_{t}=0\} \in \mathcal{A}_{G,\psi}$,
i.e., $\{v_{s},v_{t}\} \in E_{G}\ (s<t)$.
We will verify $ \theta_{k}(x_{s}-x_{t}) \in (x_{s}-x_{t})S $. 
From the definition of $C_{\geq k}$,
we can immediately check that
$s \in C_{\geq k}$ implies $t \in C_{\geq k}$.
Thus we only need to consider the cases of
(i) $s,t \in C_{\geq k}$,
(ii) $s \notin C_{\geq k}$ and $t \in C_{\geq k}$, and
(iii) $s,t \notin C_{\geq k}$.
When (i) $s,t \in C_{\geq k}$, then
\begin{align*}
&\theta_{k}(x_{s}-x_{t}) \\
&=
\prod_{j \in E_{< k}}
(x_{j}-x_{s})
\prod_{a \in \psi (v_{k})}
(x_{s}-az) 
-\prod_{j \in E_{< k}}
(x_{j}-x_{t})
\prod_{a \in \psi (v_{k})}
(x_{t}-az) \\
&\equiv
\prod_{j \in E_{< k}}
(x_{j}-x_{s})
\prod_{a \in \psi (v_{k})}
(x_{s}-az) 
-\prod_{j \in E_{< k}}
(x_{j}-x_{s})
\prod_{a \in \psi (v_{k})}
(x_{s}-az) \\
&\hspace{24em}
(\mathrm{mod}\ x_{s}-x_{t}) \\
&=0.
\end{align*}
When (ii) $s \notin C_{\geq k}$ and $t \in C_{\geq k}$,  there is a path 
$v_{k}v_{j_{1}} \cdots v_{j_{n}} v_{t}$
such that
$k < j_{1} < \cdots <j_{n} < t$.
Since $\{v_{j_{n}}, v_{t}\}$, $\{v_{s}, v_{t}\} \in E_{G}$
and $v_{t}$ is simplicial in the subgraph induced by the vertices
$\{ v_{j} \mid j \leq t \}$,
we have $\{v_{s}, v_{j_{n}}\} \in E_{G}$.
Continuing this process, we can see that
$\{v_{s},v_{k}\} \in E_{G}$,
hence $s \in E_{< k}$.
Thus
\[
\theta_{k}(x_{s}-x_{t}) 
=-\prod_{j \in E_{< k}}
(x_{j}-x_{t})
\prod_{a \in \psi (v_{k})}
(x_{t}-az)
\in (x_{s}-x_{t})S.
\]
When (iii) $s,t \notin C_{\geq k}$, then
$\theta_{k}(x_{s}-x_{t})=0$.
Therefore, we conclude that $\theta_{k}(x_{s}-x_{t}) \in (x_{s}-x_{t})S$
for any $\{ x_{s}-x_{t}=0 \} \in \mathcal{A}_{G,\psi}$.

Suppose that $\{x_{s}-bz=0\} \in \mathcal{A}_{G,\psi}$,
i.e., $b \in \psi (v_{s})$.
If $s \in C_{\geq k}$,
then there is a path 
$v_{k}v_{j_{1}} \cdots v_{j_{n}} v_{s}$
such that
$k < j_{1} < \cdots <j_{n} < s$, which implies 
$\psi (v_{k}) \supseteq \psi (v_{j_{1}}) \supseteq \cdots
\supseteq \psi (v_{j_{n}}) \supseteq \psi (v_{s})$.
Thus $b \in \psi (v_{k})$, so
\[
\theta_{k}(x_{s}-bz)=
\prod_{j \in E_{< k}}
(x_{j}-x_{s})
\prod_{a \in \psi (v_{k})}
(x_{s}-az)
\in (x_{s}-bz)S.
\]
If $s \notin C_{\geq k}$, then $\theta_{k}(x_{s}-bz)=0$.
Therefore $\theta_{k}(x_{s}-bz) \in (x_{s}-bz)S$
for $\{ x_{s}-bz=0 \} \in \mathcal{A}_{G,\psi}$.
Consequently, $\theta_{k} \in D(\mathcal{A}_{G,\psi})$
for $1 \leq k \leq \ell$.
Finally, we show that $\theta_{E},\theta_{1},\ldots.\theta_{\ell}$
form a basis for $D(\mathcal{A}_{G,\psi})$.
Notice that if $1 \leq i <k$,
then $i \notin C_{\geq k}$, so $\theta_{k}(x_{i})=0$.
Thus the determinant of the coefficient matrix of
$\theta_{E},\theta_{1},\ldots.\theta_{\ell}$
can be calculated as follows:
\begin{align*}
\left|
\begin{array}{cccc}
\theta_{E}(z) & \theta_{1}(z) & \cdots & \theta_{\ell}(z) \\
\theta_{E}(x_{1}) & \theta_{1}(x_{1}) & \cdots & \theta_{\ell}(x_{1}) \\
\vdots & \vdots & \cdots & \vdots \\
\theta_{E}(x_{\ell}) & \theta_{1}(x_{\ell}) & \cdots & \theta_{\ell}(x_{\ell})
\end{array}
\right| 
&=
\left|
\begin{array}{ccccc}
z & 0 & \cdots & 0 \\
x_{1} & \theta_{1}(x_{1}) & \ddots & \vdots \\
\vdots & \vdots & \ddots & 0 \\
x_{\ell} & \theta_{1}(x_{\ell}) & \cdots & \theta_{\ell}(x_{\ell}) 
\end{array}
\right| \\
&=z \prod_{k=1}^{\ell} \theta_{k}(x_{k}) \\
&=z
\prod_{k=1}^{\ell}
\left(
\prod_{j \in E_{< k}}
(x_{j}-x_{k})
\prod_{a \in \psi (v_{k})}
(x_{k}-az)
\right) \\
&=z
\prod_{v_{j}v_{k} \in E_{G}}
(x_{j}-x_{k})
\prod_{\substack{1\leq k \leq \ell \\ a \in \psi (v_{k})}}
(x_{k}-az) \\
&=Q(\mathcal{A}_{G,\psi}).
\end{align*}
Therefore it follows from Saito's criterion \cite{saito1980theory} (see also {\cite[Theorem 4.19]{orlik1992arrangements}})
that $\theta_{E},\theta_{1},\ldots.\theta_{\ell}$
form a basis for $D(\mathcal{A}_{G,\psi})$.
\end{proof}

Combining the proof of Theorem \ref{basis} and 
following Ziegler's theorem,
we may derive that $\mathcal{A}_{G,\psi}$ is supersolvable
if $ (G,\psi) $ has a weighted elimination ordering. 
This is another proof of the sufficiency in Theorem \ref{thm:MS1}. 

\begin{theorem}[Ziegler {\cite[Theorem 6.6]{ziegler1989combinatorial}}]
Let $\mathcal{A}$ be a central arrangement.
Then $\mathcal{A}$ is supersolvable
if and only if
there is a basis for $D(\mathcal{A})$
such that the coefficient matrix of the basis is lower triangular
for some choice of coordinates. 
\end{theorem}

As a corollary of Theorem \ref{basis}, we may give the exponents of $ \mathcal{A}_{G,\psi} $, which was stated by Mu and Stanley without proof. 

\begin{corollary}[Mu-Stanley {\cite[Proposition 2]{mu2015supersolvability}}]\label{cor:MS3}
Suppose that $(G,\psi)$ has a weighted elimination ordering $ (v_{1},\ldots,v_{\ell}) $
and let $E_{<k}$ be as in Theorem \ref{basis}.
Then $\mathcal{A}_{G,\psi}$ is supersolvable with exponents
$\{ 1,|E_{<1}|+|\psi(v_{1})|,\ldots, |E_{<\ell}|+|\psi(v_{\ell})| \}$.
\end{corollary}

\section{The proof of Theorem \ref{thm:main1}}\label{sec:proof}
We will first prove the following lemma. 

\begin{lemma}\label{lem:nonfreepath}
Let $ \ell \geq 3 $ and $ \mathcal{A} = \mathcal{A}_{G, \psi} $ a $ \psi $-graphical arrangement defined by the following graph $ (G, \psi) $: 
\begin{center}
\begin{tikzpicture}[baseline=0pt, 
v/.style={circle, draw, inner sep=0pt, minimum size=4pt, fill=black}
]
\draw (0,0) node[v,label=below: $ \scriptstyle 1 $] (1){};
\draw (1,0) node[v,label=below: $ \scriptstyle 2 $] (2){};
\draw (2,0) node[v,label=below: $ \scriptstyle 3 $] (3){};
\draw (2.5, 0) node[] (a){};
\draw (3.5, 0) node[] (b){};
\draw (4,0) node[v,label=below: $ \scriptstyle \ell-2 $] (l-2){};
\draw (5,0) node[v,label=below: $ \scriptstyle \ell-1 $] (l-1){};
\draw (6,0) node[v,label=below: $ \scriptstyle \ell $] (l){};
\draw (1)--(2)--(3)--(a);
\draw[dashed] (a)--(b);
\draw (b)--(l-2)--(l-1)--(l);
\end{tikzpicture} \\
with $ \psi(1) \supsetneq \psi(2) = \dots = \psi(\ell-1) \subsetneq \psi(\ell). $
\end{center}
Then $ \mathcal{A} $ is not free. 
\end{lemma}

The following theorem is required: 

\begin{theorem}[Orlik-Terao {\cite[Theorem 4.46]{orlik1992arrangements}}, Terao \cite{terao1980arrangements}]\label{thm:ad}
Let $ \mathcal{A} $ be an arrangement and $ H_{0} \in \mathcal{A} $. 
Let $ (\mathcal{A}, \mathcal{A}^{\prime}, \mathcal{A}^{\prime\prime
}) $ be the triple with respect to $ H_{0} $, i.e., $ \mathcal{A}^{\prime} = \mathcal{A} \setminus \{H \}, \mathcal{A}^{\prime\prime} = \left\{ H \cap H_{0} \mid H \in \mathcal{A}^{\prime} \right\} $. 
If $ \mathcal{A} $ and $ \mathcal{A}^{\prime} $ are free then $ \mathcal{A}^{\prime\prime} $ is also free and $ \exp \mathcal{A}^{\prime\prime} \subset \exp \mathcal{A}^{\prime} $. 
\end{theorem}

\begin{proof}[Proof of Lemma \ref{lem:nonfreepath}]
We will prove the assertion by induction on $ \ell $. 
Suppose that $ \ell = 3 $. 
Let $ H = \{ x_{2}-x_{3} = 0 \} $ and $ (\mathcal{A}, \mathcal{A}^{\prime}, \mathcal{A}^{\prime\prime}) $ the triple with respect to $ H $. 
Then the ordering $ (1,2,3) $ is a weighted elimination ordering for $ \mathcal{A}^{\prime} $. 
Hence $ \mathcal{A}^{\prime} $ is supersolvable and free with the exponents $ \exp(\mathcal{A}^{\prime}) = (1, |\psi(1)|, |\psi(2)|+1, |\psi(3)| ) $ by Corollary \ref{cor:MS3}. 
Assume that $ \mathcal{A} $ is free.
Then Theorem \ref{thm:ad} asserts that $ \mathcal{A}^{\prime\prime} $ is also free and $ \exp \mathcal{A}^{\prime\prime} \subset \exp \mathcal{A}^{\prime} $. 
Then $ \psi(1) \subseteq \psi(3) $ or $ \psi(1) \supseteq \psi(3) $ by Lemma \ref{lem:MS2}. 
By Lemma \ref{cor:MS3}, the exponents of $ \mathcal{A}^{\prime\prime} $ are $ (1, |\psi(1)|+1, |\psi(3)|) $ or $ (1, |\psi(1)|, |\psi(3)|+1) $ respectively. 
In the both case, we have that $ \exp \mathcal{A}^{\prime\prime} \not\subset \exp \mathcal{A}^{\prime} $. 
This contradiction concludes that $ \mathcal{A} $ is not free. 

Suppose that $ \ell \geq 4 $ and that the statement is true for $ \ell-1 $. 
Let $ H = \{x_{\ell-1}-x_{\ell}=0 \} $ and $ (\mathcal{A}, \mathcal{A}^{\prime}, \mathcal{A}^{\prime\prime}) $ the triple with respect to $ H $. 
Since the ordering $ (1,2, \dots, \ell) $ is a weighted elimination ordering for $ \mathcal{A}^{\prime} $, the arrangement $ \mathcal{A}^{\prime} $ is free. 
If $ \mathcal{A} $ is free, then $ \mathcal{A}^{\prime\prime} $ is also free by Theorem \ref{thm:ad}. 
Contrary, by the induction hypothesis, we have that $ \mathcal{A}^{\prime\prime} $ is not free, which is a contradiction. 
Hence $ \mathcal{A} $ is not free. 
\end{proof}

\begin{proof}[Proof of Theorem \ref{thm:main1}]
By Theorem \ref{thm:MS1}, the equivalence $ (1) \Leftrightarrow (2) $ holds. 
The implication $ (2) \Rightarrow (3) $ is well-known. 
Theorem \ref{thm:vwg} shows $ (1) \Leftrightarrow (4) \Leftrightarrow (5) $. 
The rest part of the proof is $ (3) \Rightarrow (4) $. 
Suppose that $ \mathcal{A}_{G, \psi} $ is free. 
By Proposition \ref{prop:local}, the graphical arrangement $ \mathcal{A}_{G} $ is free. 
Therefore $ G $ is chordal by Theorem \ref{thm:stanley}. 
By Lemma \ref{lem:MS2}, Proposition \ref{prop:local} and Lemma \ref{lem:nonfreepath}, $ (G,\psi) $ has no induced paths in (4). 
Thus the proof has been completed. 
\end{proof}

\section{Characterization of the freeness of the $\psi$-graphical multiarrangements}\label{sec:multi}
In this section, we introduce the $ \psi $-graphical multiarrangements, which are determined by a vertex-weighted graph over $ \mathbb{Z}_{\geq 0} $, and give a characterization of the freeness of the $\psi$-graphical multiarrangements.
First, we introduce some definitions of multiarrangements.

A \textbf{multiarrangement} $(\mathcal{A},m)$ is a pair of
an arrangement $\mathcal{A}$ with a map
$m:\mathcal{A} \rightarrow \mathbb{Z}_{\geq 0}$,
which is called multiplicity, and its defining polynomial is
$Q(\mathcal{A},m)=\prod_{H \in \mathcal{A}} \alpha_{H}^{m(H)}$.
When $m(H)=1$ for any $H \in \A$, then $(\A,m)$ is just a hyperplane arrangement and sometimes called a simple arrangement.
The module of logarithmic derivations $D(\mathcal{A},m)$
of a multiarrangement is defined by
\[
D(\mathcal{A},m):=
\{ \theta \in \Der(S) \mid \theta (\alpha_{H}) \in
\alpha_{H}^{m(H)} S \text{ for any } H \in \mathcal{A} \}.
\]
We say that $(\mathcal{A},m)$ is \textbf{free} if $D(\mathcal{A},m)$
is a free $S$-module.

\begin{definition}
Let $ (G,\psi) $ be a vertex-weighted graph over $ \mathbb{Z}_{\geq 0} $ with vertex set $ V_{G}=\{1,\dots,\ell\} $ and edge set $ E_{G} $. 
Note that $ \psi $ is a map from $ V_{G} $ to $ \mathbb{Z}_{\geq 0} $. 
The \textbf{$\psi$-graphical multiarrangement} $\pgma$ is defined by the following defining polynomial: 
\[
Q(\pgma) = \prod_{\{i,j\} \in E_{G}} (x_{i}-x_{j})
\prod_{1 \leq i \leq \ell} x_{i}^{\psi(i)}.
\] 
\end{definition}

We can regard a $\psi$-graphical multiarrangement as a graphical multiarrangement, which is a graphical arrangement with multiplicity, by the coordinates change $x_{i} \mapsto x_{0}-x_{i}\ (1 \leq i \leq \ell)$. 
Several partial results for characterizing the freeness of graphical multiarrangements have been found, for example, see \cite{abe2009signed,DiPasquale2017generalized,terao2002multi}.

The freeness of the $\psi$-graphical multiarrangements is characterized as follows similarly to Theorem \ref{thm:main1}:

\begin{theorem}\label{thm:multi}
Let $\pgma$ be a $\psi$-graphical multiarrangement.
Then the following conditions are equivalent: 
\begin{enumerate}
\item $ (G, \psi) $ has a weighted elimination ordering. 
\item $ D( \pgma ) $ has a basis whose coefficient matrix is lower triangular.
\item $ \pgma $ is free. 
\end{enumerate}
\end{theorem}
To prove this theorem, we need the following theorems and lemmas.

Let $\mathcal{A}$ be a central arrangement and fix $H_{0} \in \mathcal{A}$.
The \textbf{Ziegler restriction} $(\mathcal{A}^{H_{0}},m^{H_{0}})$
of $\mathcal{A}$ onto $H_{0}$ is defined by
$\mathcal{A}^{H_{0}} =
\{ H \cap H_{0} \mid H \in \mathcal{A} \setminus \{H_{0}\} \}$,
$m^{H_{0}} (X) =
| \mathcal{A}_{X} \setminus \{ H_{0} \} |
\text{ for } X \in \mathcal{A}^{H_{0}}.$
\begin{theorem}[Ziegler {\cite[Theorem 11]{ziegler1989multiarrangements}}]\label{ziemulti}
Let $\mathcal{A}$ be an arrangement and $H_{0} \in \mathcal{A}$.
If $\mathcal{A}$ is free and $\{ \theta_{E}, \theta_{1}, \ldots, \theta_{\ell} \} $ is a basis for $D(\mathcal{A})$, 
where $\theta_{i}(\alpha_{H_{0}})=0$ for $1 \leq i \leq \ell$,
then the Ziegler restriction $(\mathcal{A}^{H_{0}},m^{H_{0}})$ is free
and $\{ \theta_{1}|_{H_{0}}, \ldots, \theta_{\ell}|_{H_{0}} \} $
is a basis for $D(\mathcal{A}^{H_{0}},m^{H_{0}})$.
\end{theorem}
\begin{theorem}[Yoshinaga {\cite[Theorem 2.2]{yoshinaga2004characterization}}]\label{yoshicriterion}
Suppose $\ell \geq 3$.
Let $\A$ be an arrangement in $(\ell+1)$-dimensional vector space and $H_{0} \in \A$.
Then $\A$ is free if and only if the Ziegler restriction $(\A^{H_{0}},m^{H_{0}})$ is free and $\A_{X}$ is free for any $X \in L(\A^{H_{0}}) \setminus \{ \bigcap_{H \in \A} H \} $.
\end{theorem}
\begin{lemma}[Abe {\cite[Lemma 3.8]{abe2006freeness}}]\label{multilocal}
Let $(\A,m)$ be a multiarrangement and $X \in L(\A)$.
If $(\A,m)$ is free, then the localization $(\A_{X},m_{X})$ is free, where $m_{X}(H)=m(H)$ for $H \in \A_{X}$.
\end{lemma}
\begin{lemma}[Terao {\cite[Lemma 1]{terao1983free-potjasams}}, Ziegler {\cite[Theorem 1.5]{ziegler1990matroid-totams}}]\label{fieldextension}
Let $ (\mathcal{A},m) $ be a multiarrangement in a vector space $V$ over a field $ \mathbb{K} $. 
For any field extension $ \mathbb{F}/\mathbb{K} $, $ (\mathcal{A},m) $ is free if and only if $ (\mathcal{A}_{\mathbb{F}}, m_{\mathbb{F}}) $ is free, where $ \mathcal{A}_{\mathbb{F}} $ denotes the arrangement in $V \otimes_{\mathbb{K}} \mathbb{F}$ consisting of hyperplanes $ H \otimes_{\mathbb{K}} \mathbb{F} $ for $ H \in \mathcal{A} $ and $ m_{\mathbb{F}}(H \otimes_{\mathbb{K}} \mathbb{F}) \coloneqq m(H) $. 
\end{lemma}
\begin{lemma}\label{localinduced}
Suppose a graph $G$ is a path $ v_{1} \cdots v_{k} $ with $ k \geq 3 $ and $\psitilde  \colon V_{G} \rightarrow 2^{\mathbb{K}}$ satisfies $\psitilde(v_{1}) \supsetneq \psitilde(v_{2}) = \dots = \psitilde(v_{k-1}) \subsetneq \psitilde(v_{k}) $.
Let $H_{0}=\{ z=0 \}$.
Then $(\A_{G,\psitilde})_{X}$ is free for any $X \in L(\A_{G,\psitilde}^{H_{0}}) \setminus \{ \bigcap_{H \in \A_{G,\psitilde}} H \}$. 
\end{lemma}
\begin{proof}
Since $(\A_{G,\psitilde})_{X}$ is a subarrangement of the $ \psitilde $-graphical arrangement $ \mathcal{A}_{G,\psitilde} $ containing $ H_{0} $, we may put $\A_{S,\psitilde'}=(\A_{G,\psitilde})_{X}$, where $S$ is a subgraph of $G$ and $\psitilde':V_{S} \rightarrow 2^{\mathbb{K}}$.
Moreover, we may assume that $S$ is connected.

Let $ x_{1},\dots,x_{k} $ be the coordinates corresponding to the vertices $ v_{1}, \dots, v_{k} $. 
If $\A_{S,\psitilde'}$ does not contain the hyperplanes of the form $\{ x_{i}=az \}$, then $\A_{S,\psitilde'}$ is a simple graphical arrangement of the chordal graph $ S $ with the hyperplane $ H_{0} $. 
Hence $\A_{S,\psitilde'}$ is free by Theorem \ref{thm:stanley}.

Assume that $ \mathcal{A}_{S,\psitilde'} $ contains at least one hyperplane $ \{x_{i_{0}}=a_{0}z\} $, i.e., there exist a vertex $ v_{i_{0}} \in V_{S} $ and an element $ a_{0} \in \psitilde'(v_{i_{0}}) $. 
Since $ H_{0}=\{z=0\} \in \mathcal{A}_{S,\psitilde'} $, we have $ \{x_{i_{0}}=az\} \in \mathcal{A}_{S,\psitilde'} $ for any $ a \in \psitilde(v_{i}) $. 
Moreover, since $ S $ is connected, we have $ \{x_{i}=az\} \in \mathcal{A}_{S,\psitilde'} $ for any $ v_{i} \in V_{S} $ and $ a \in \psitilde(v_{i}) $. 
Therefore $\psitilde'=\psitilde_{S}$, where $\psitilde_{S} \coloneqq \psitilde|_{V_{S}}$. 
Since $X \neq \bigcap_{H \in \A_{G,\psitilde}} H$, we have $\A_{S,\psitilde'}=(\A_{G,\psitilde})_{X} \subsetneq \mathcal{A}_{G,\psitilde}$, namely $S \subsetneq G$.  Hence $(S,\psitilde_{S})$ satisfies the condition in Theorem \ref{thm:main1} (4).
Therefore $\A_{S,\psitilde'}=\A_{S,\psitilde_{S}}$ is free.
\end{proof}

\begin{proof}[Proof of Theorem \ref{thm:multi}]
By Lemma \ref{fieldextension}, we may assume that the ground field $ \mathbb{K} $ of $ \mathcal{M}_{G,\psi} $ is infinite. 
Then there exists an injection $ \iota \colon \mathbb{Z}_{> 0} \rightarrow \mathbb{K} $. 
Define $\psitilde \colon V_{G} \rightarrow 2^{\mathbb{K}}$ by $\widetilde{\psi}(i)=\{ \iota(1), \dots, \iota(\psi(i)) \}$. 
Then $ \mathcal{M}_{G,\psi} $ is the Ziegler restriction of the $ \psitilde $-graphical arrangement $ \mathcal{A}_{G,\psitilde} $ on to the hyperplane $ \{z=0\} $. 

Suppose that $ (G,\psi) $ has a weighted elimination ordering. 
Then $ (G,\psitilde) $ also has a weighted elimination ordering. 
Hence, by Theorem \ref{basis}, $ D(\mathcal{A}_{G,\psitilde}) $ has a basis $\{ \theta_{E},\theta_{1},\ldots,\theta_{\ell} \}$ whose coefficient matrix is lower triangular and $ \theta_{i}(z) = 0 $ for $ 1 \leq i \leq \ell $. 
Therefore, we can see $(1) \Rightarrow (2)$ from Theorem \ref{ziemulti}.

The implication $(2) \Rightarrow (3)$ is obvious. 

Let us show $(3) \Rightarrow (1)$. 
Suppose that $ \mathcal{M}_{G,\psi} $ is free. 
It is easy to verify that the graphical arrangement $ \mathcal{A}_{G} $ is a localization of $ \mathcal{M}_{G,\psi} $ by Proposition \ref{prop:local}. 
By Lemma \ref{multilocal}, we have $ \mathcal{A}_{G} $ is free, and hence $ G $ is chordal by Theorem \ref{thm:stanley}. 

Assume that there exists an induced path $ P = v_{1} \cdots v_{k} $ of $ G $ with $ k \geq 3 $ and $ \psi(v_{1}) > \psi(v_{2}) = \dots = \psi(v_{k-1}) < \psi(v_{k}) $. 
Note that $ \mathcal{M}_{P,\psi_{P}} $ is a localization of $ \mathcal{M}_{G,\psi} $ by Proposition \ref{prop:local}. 
Hence it follows from Lemma \ref{multilocal} that $ \mathcal{M}_{P,\psi_{P}} $ is free. 
However, Using Lemma \ref{localinduced} and Theorem \ref{yoshicriterion}, we have $ \mathcal{M}_{P,\psi_{P}} $ is not free since $ \mathcal{A}_{P,\widetilde{\psi_{P}}} $ is not free by Theorem \ref{thm:main1}. 
This is a contradiction. 
Therefore $ (G,\psi) $ does not contain such induced paths. 

Since $ \mathbb{Z}_{\geq 0} $ is totally ordered, we conclude that $ (G,\psi) $ has a weighted elimination ordering by Theorem \ref{thm:vwg}. 
\end{proof}

\section*{Acknowledgments}
The authors would like to express the deepest appreciation to Takuro Abe whose comments made enormous contribution to our work.

\bibliographystyle{amsplain1}
\bibliography{bibfile}

\end{document}